\documentclass[12pt,reqno,twoside]{amsart}
\usepackage{amsmath}
\usepackage{amssymb}
\usepackage{color}
\usepackage{hyperref}
\usepackage[english]{babel}

\title{New bounds on the density of lattice coverings}

\author{Or Ordentlich}
\address{School of Computer Science and Engineering, Hebrew University}

\author{Oded Regev}
\address{Courant Institute of Mathematical Sciences, New York University}

\author{Barak Weiss}
\address{School of Mathematical Sciences, Tel Aviv University}

\font\sb = cmbx8 scaled \magstep0
\font\sn = cmssi8 scaled \magstep0
\font\si = cmti8 scaled \magstep0
\long\def\comoded#1{\ifdraft{{\color{blue}\si Oded: #1 }}\else\ignorespaces\fi}
\long\def\combarak#1{\ifdraft{\color{red}\sn Barak: #1 }\else\ignorespaces\fi}
\long\def\comor#1{\ifdraft{\color{cyan}\sb Or: #1  }\else\ignorespaces\fi}

\numberwithin{equation}{section}

\newcommand{\KK}{\mathcal{K}}
\newcommand{\NN}{\mathcal{N}}

\newif\ifdraft\drafttrue
\draftfalse
\newcommand\name[1]{\label{#1}{\ifdraft{\sn [#1]}\else\ignorespaces\fi}}

\newcommand\eq[2]{{\ifdraft{\ \tt [#1]}\else\ignorespaces\fi}\begin{equation}\label{#1}{#2}\end{equation}}
\newcommand {\equ}[1]{\eqref{#1}}

\newcommand{\R}{{\mathbb{R}}}
\newcommand{\TT}{{\mathbb{T}}}

\newcommand{\Z}{{\mathbb{Z}}}

\newcommand{\BB}{{\mathcal{B}}}
\newcommand{\SSS}{{\mathcal{S}}}
\newcommand{\AAA}{{\mathcal{A}}}
\newcommand{\Vol}{{\mathrm{Vol}}}
\newcommand{\crog}{{c_{\mathrm{Rog}}}}
\newcommand{\prob}{{\Pr}}

\newcommand{\Gr}{{\mathrm{Gr}}}

\newcommand{\N}{{\mathbb{N}}}

\newcommand{\F}{{\mathbb{F}}}

\newcommand{\GL}{\operatorname{GL}}

\newcommand{\SL}{\operatorname{SL}}

\newcommand {\ignore}[1]  {}

\newcommand{\spa}{{\rm span}}

\newcommand{\PP}{{\mathcal P}}
\newcommand{\PPdisc}{{\mathcal P^{\mathrm{(disc)}}_{L}}}
\newcommand{\PPdiscbar}{{\overline{\mathcal P}^{\mathrm{(disc)}}_{L}}}

\newcommand{\LL}{{\mathcal L}}

\newcommand{\df}{{\, \stackrel{\mathrm{def}}{=}\, }}

\newcommand{\sm}{\smallsetminus}

\newcommand{\vre}{\varepsilon}

\newcommand{\Convn}{{\mathrm{Conv}_n}}

\newtheorem{thm}{Theorem}[section]

\newtheorem{lem}[thm]{Lemma}

\newtheorem{prop}[thm]{Proposition}

\newtheorem{cor}[thm]{Corollary}

\newtheorem{remark}[thm]{Remark}

\begin{document}
\date{\today}
\maketitle

\begin{abstract}
We obtain new upper bounds on the minimal density $\Theta_{n, \KK}$ of lattice
coverings of $\R^n$ by dilates of a convex body $\KK$. We also obtain
bounds on the probability (with respect to the natural Haar-Siegel
measure on the space of lattices) that a randomly chosen lattice $L$
satisfies $L+\KK = \R^n$. As a step in the proof, we utilize and strengthen
results on the discrete Kakeya problem.

\end{abstract}

\section{Introduction}
The classical lattice covering problem asks for the most economical way
to cover space by overlapping Euclidean balls centered at points of a
lattice. To make this precise, given a lattice $L \subset \R^n$,
normalized so that it has covolume one, 
define its {\em covering density}, denoted $\Theta(L)$, to be the minimal volume of a closed
Euclidean ball $B_r$, for which  $\R^n = L + B_r.$ Define
$$\Theta_n \df \inf\{\Theta(L) : L \text{ is a lattice of covolume one
  in } \R^n \}.$$
Similarly, let $\KK \in \Convn$, where $\Convn$
denotes the set of 
compact convex subsets of  $\R^n$ with nonempty interior. We define
the {\em $\KK$-covering density} of $L$,
denoted $\Theta_\KK(L)$, to be the minimal volume of a dilate $r
\cdot \KK$  such that $\R^n = L + r \cdot \KK$, and define
$$\Theta_{n, \KK} \df \inf\{\Theta_\KK(L) : L \text{ is a lattice of covolume one
  in } \R^n \}.
$$
The quantities $\Theta_n$ and $\Theta_{n , \KK}$ have been intensively
investigated, both for individual $n$ and $\KK$, and asymptotically
for large $n$, and many questions remain open. 
Standard references are \cite{ConwaySloane, GL, Rogers_pandc}.

The collection $\LL_n$ of lattices of covolume one in $\R^n$ can be identified
with the quotient $\SL_n(\R)/\SL_n(\Z)$, via the map
\eq{eq: mapsto}{
  g \SL_n(\Z) \mapsto g \Z^n \ \  (g \in \SL_n(\R)).}
This identification endows $\LL_n$ with a natural probability
measure; namely, there is a unique $\SL_n(\R)$-invariant Borel
probability measure on $\LL_n$. We will refer to this measure as the
{\em Haar-Siegel measure} and denote it by
$\mu_n$.
In this paper we give new bounds on $\Theta_{n, \KK}$ and 
on the $\mu_n$-typical value of
$\Theta_{\KK}(L)$.

\begin{thm}\name{thm: improving Rogers}
  There is $c>0$ so that for any $n \in \N$ and any $\KK \in \Convn$,
 \eq{eq: improving Rogers}{
\Theta_{n, \KK} \leq c n^2.
}
\end{thm}

This improves on the best previous bound of $n^{\log_2 \log n + c},$
which was proved by  Rogers \cite{Rogers_again}. 
We note that 
for the case that $\KK$ is the Euclidean ball,
Rogers obtained $\Theta_n \leq n 
\, (\log n)^c$~\cite{Rogers_again}, and this was extended to certain symmetric convex
bodies by Gritzmann \cite{Gritzmann}. This bound is better than what
we obtain here.

We will actually prove the following measure estimate, from which Theorem~\ref{thm: improving Rogers} follows immediately. 
  
\begin{thm}\name{thm: main} There are positive constants $c_1, c_2,
  c_3, c_4$ such that for any $n \in \N$, 
  any $\KK \in \Convn$, and any 
  \eq{eq: constraints M}{
    M 
    \in \left[c_3 n^2 , c_4 n^3\right],
  }
   we have 
\eq{eq: conclusion main}{
\mu_n\left( \left\{L \in \LL_n : \Theta_{\KK}(L)>M \right \} \right) <
c_1 \, e^{-\frac{c_2 M}{n^2}}.
  }
\end{thm}
We remark that the constants appearing in the 
statement of Theorem~\ref{thm: main} can be explicitly estimated.

\begin{remark}
As we will show in Appendix \ref{appendix B}, 
the left-hand side of \equ{eq: conclusion main} is at least $C/M$, for some constant $C$ depending on $n$ and $\KK$. It follows that some upper 
bound on $M$ is required if \equ{eq: conclusion main} is to hold. It also follows that the expectation of $\Theta_\KK$ with respect to the measure $\mu_n$ is infinite. 
\end{remark}




Setting $M = c_4 n^3$ in \equ{eq: conclusion
  main}, we see that

\begin{cor}\name{cor: improving Strombergsson}
There is a constant $c>0$ such that 
  for any 
  sequence $\KK_n \in \Convn$, the
Haar-Siegel probability that $\Theta_{n,\KK_n} (L)
\leq cn^3$ tends to 1 exponentially fast, as $n \to \infty$. 
\end{cor}

This resolves a question of Str\"ombergsson, who showed in 
\cite{Strombergsson} that the conclusion holds with
$\Theta_{n,\KK_n} (L)
\leq (1+\delta)^n$ and $\delta >
\delta_0$, for an explicit number 
$\delta_0 = 0.756...$.

We introduce two quantities which describe the growth rate of the
Haar-Siegel typical covering density. Let 
$$
\tau_{\circ} \df \inf \left\{s> 0 : \mu_n 
    \left\{L \in \LL_n : \Theta(L) < n^{s}
  \right\} 
\longrightarrow_{n \to \infty} 1 \right \}
$$
and 
$$
\tau \df \inf \left \{s> 0 : \inf_{\KK \in \Convn} \, \mu_n
    \left\{L \in \LL_n : \Theta_{\KK}(L) < n^{s}
  \right\} 
\longrightarrow_{n \to \infty} 1 \right \}.
$$
Clearly $\tau_\circ \leq \tau$, and a result of Coxeter,
Few and Rogers \cite{CoxeterFewRogers} implies that 
$\tau_\circ \geq 1$. 
Plugging $M = n^{2 + \vre}$ into \equ{eq: conclusion main} we deduce
the following.

\begin{cor}\name{cor: tau min}
We have $\tau \leq 2.$ 
\end{cor}
Prior to our
results it was not known whether $\tau$ and $\tau_\circ$ are finite,
i.e., whether 
the typical behavior of the covering density is polynomial. It would
be interesting to know whether our 
upper bound on $\tau$ and $\tau_\circ$ can be improved.

\subsection{Simultaneous covering and packing}
We describe another application of Theorem~\ref{thm: main}, improving a
result of Butler \cite{Butler}. 
To state it, define
the {\em $\KK$-packing density} of $L$,
denoted $\delta_\KK(L)$, to be the maximal volume of a dilate $r
\cdot \KK$  such that the translates $\{\ell +  r \cdot \KK : \ell \in L\}$ are disjoint.
Then we have:
\begin{cor}\name{cor: improving Butler}
  There is $c>0$ such that for all $n \in \N$ and all $\KK \in
  \Convn$, there is $L \in \LL_n$ such that
  \eq{eq: improving Butler}
  {
    \frac{\Theta_\KK(L)}{\delta_\KK(L)}
  \leq  
    c \, \frac{\mathrm{Vol}(\KK - 
    \KK)}{\mathrm{Vol}(\KK)}  \, n^2 \; .
    }
  \end{cor}
This improves a previous upper bound of 
  $(\mathrm{Vol}(\KK -  
    \KK)/\mathrm{Vol}(\KK)) n^{\log_2 \log n + c}$ 
proved by Butler~\cite{Butler}. 
The proof follows by observing that (a) a dilate of volume  
$\ll \mathrm{Vol}(\KK)/\mathrm{Vol}(\KK-\KK)$ is 
with high probability packing for a Haar-Siegel random $L$ (by Siegel's theorem~\cite{SiegelFormula}),
and that (b) a dilate of volume 
$\gg n^2$ is 
with high probability covering for a Haar-Siegel random $L$ (by Theorem~\ref{thm: main}).
The union bound then shows that with high probability both events 
hold simultaneously, completing the proof. 
We leave the details to the reader. 

The fact that~\equ{eq: improving Butler} holds with high probability for a $\mu_n$-random lattice can be used to derive the following strengthening. 
Since $\mu_n$ is preserved by the mapping which sends $L$ to its dual $L^*$ (see \equ{eq: dual lattice}), we obtain the existence of a lattice $L$ such that both $L$ and $L^*$ satisfy~\equ{eq: improving Butler}.




\subsection{Ingredients of the proof} Our proof of Theorem \ref{thm: main} utilizes some lower bounds
on the cardinality of discrete Kakeya sets (see \S \ref{subsec:
  Kakeya}). Specifically, 
relying on a result of Kopparty, Lev, Saraf,
and Sudan \cite{KLSS}, we obtain a new
lower bound on the size of a 
discrete $\vre$-Kakeya set of rank 2, see Corollary \ref{cor: bound
  eps Kakeya}. What is
important for us is that the dependence of this bound on the parameter
$\vre$ is linear.

We also use a
variant of the Hecke correspondence to analyze the properties of a
$\mu_n$-typical lattice. Namely, we show in \S \ref{subsec: Haar
  Siegel} that for parameters $p, r$, if one draws a Haar-Siegel random lattice
$L$, and 
then replaces it by a lattice $L'$ uniformly drawn from those
containing $L$ as a sub-lattice 
of index $p^r$, and with a prescribed quotient group $L'/L$, then $L'$
(properly rescaled) is also Haar-Siegel random.  Our construction is
inspired by a similar construction which was
investigated by Erez, Litsyn and Zamir \cite{Urietal} in the
information theory literature.

\subsection{Acknowledgements}
We are grateful to Uri Erez, Swastik Kopparty, and Alex Samorodnitsky for useful discussions.  
The authors gratefully acknowledge the support of grants ISF 2919/19,
ISF 1791/17,
BSF 2016256, the Simons Collaboration on Algorithms and Geometry, a
Simons Investigator Award, and by the National Science Foundation (NSF)
under Grant No.~CCF-1814524.

\section{Preliminaries}
\subsection{Space of lattices and Haar-Siegel measure}\name{subsec:
  Haar Siegel}
Recall from the introduction that $\LL_n \cong
\SL_n(\R)/\SL_n(\Z)$. This space is endowed with the quotient topology
and hence with the Borel $\sigma$-algebra arising from this
topology. The group  
$\SL_n(\R)$ acts naturally on lattices via the linear action of
matrices on $\R^n$, or equivalently, by left translations on the
quotient $\SL_n(\R)/\SL_n(\Z)$. 
The measure 
$\mu_n$ is the unique Borel probability 
measure on $\LL_n$ which is invariant under this action. From generalities on
coset spaces of Lie groups, such a measure exists and is unique, see 
e.g., \cite{Raghunathan}. 
We will also consider a slightly more general
space, namely for each $c>0$ we write $\LL_{n,c}$ for the collection
of lattices of covolume $c$ in $\R^n$. The obvious rescaling
isomorphism 
$\LL_n \cong \LL_{n,c}$ commutes with the $\SL_n(\R)$-action,
and thus there is a unique $\SL_n(\R)$-invariant 
measure on $\LL_{n,c}$, and we will denote it by $\mu_{n,c}$. We will
refer to any of the measures $\mu_n, \mu_{n,c}$ as the 
  Haar-Siegel measure. 

For a prime $p$ and an integer $r \in \{1, \ldots, n\}$,
associate to each lattice $L$ the finite
collection $\Lambda_{p,r}(L)$ of lattices $L'$ in $\R^n$ which contain
$L$ as a sub-lattice, and for which the quotient $L'/L$ is isomorphic to
$\prod_1^r \Z/p\Z$. Note that these lattices
are of covolume $p^{-r}$.
The assignment $L \to \Lambda_{p,r}(L)$ is a particular case of the
so-called {\em Hecke correspondence} (see
e.g., \cite{COU}). The following useful observation is
well-known, we include a proof for completeness.

\begin{prop}\name{prop: Hecke} For each $n,p,r$ as above, let $N =
  |\Lambda_{p,r}(\Z^n)|$. Then for each
$f \in C_c(\LL_{n, p^{-r}})$, i.e., each continuous compactly supported
real valued function on $\LL_{n, p^{-r}}$,
  \eq{eq: Hecke}{
\int f \, d\mu_{n, p^{-r}}  = \int_{\LL_n} \frac{1}{N} \sum_{L' \in
  \Lambda_{p,r}(L)} f(L') \, d\mu_n(L).
  }
  In other words,
  choosing $L'$
  randomly according to Haar-Siegel 
  measure on $\LL_{n,p^{-r}}$ is the same as choosing $L$ randomly
  according to Haar-Siegel measure on $\LL_{n}$, and then choosing $L'$
  uniformly in $\Lambda_{p,r}(L)$.
\end{prop}

\begin{proof}
The right-hand side of \equ{eq: Hecke} describes a positive continuous
linear functional on $C_c(\LL_{n, p^{-r}})$, and hence, by the Riesz
representation theorem, is equal to $\int f \, d\nu$ for some Radon measure
$\nu$ on
$\LL_{n, p^{-r}}$. Taking a monotone increasing sequence of compactly
supported functions tending everywhere to 1, we see that \equ{eq:
  Hecke} holds for the function $f \equiv 1$, and
from this it follows that $\nu$ is a probability measure. By the
uniqueness property of Haar-Siegel measure, in order to show that $\nu
= \mu_{n, p^{-r}}$ it suffices to show that $\nu$ is invariant under
left-multiplication by any $g \in \SL_n(\R)$. From the definition of
$\Lambda_{p,r}(L)$ we see that $\Lambda_{p,r}(gL) =
g\Lambda_{p,r}(L)$, and so the invariance of $\nu$ follows from 
the
 following computation:
 \[\begin{split}
\int f \circ g \, d\nu & =  \int_{\LL_n} \frac{1}{N} \sum_{L' \in
  \Lambda_{p,r}(L)} f(g L') \, d\mu_n(L) \\
& =  \int_{\LL_n} \frac{1}{N} \sum_{L'' \in
  \Lambda_{p,r}(gL)} f(L'') \, d\mu_n(L) \\
& =  \int_{\LL_n} \frac{1}{N} \sum_{L'' \in
  \Lambda_{p,r}(gL)} f(L'') \, d\mu_n(gL) = \int f \, d\nu.
     \end{split}\]
  \end{proof}

We now
interpret this in 
terms of the discrete Grassmannian, as
follows. 
For a prime $p$ let $\F_p$ denote the field with $p$ elements. For $r
\in \{1, \ldots, n\}$, let 
$\Gr_{n,r}(\F_p)$ denote the collection of subspaces of
dimension $r$ in $\F_p^n$, or equivalently, the rank-$r$ additive
subgroups of $\F_p^n$. We can identify $\F_p$ with the residues 
$\{0, \ldots, p-1\}$, and thus identify $\F_p^n$ with the quotient
$\Z^n/p\Z^n$. We have a natural {\em reduction mod $p$} homomorphism
$\pi_p: \Z^n \to
\F_p^n$, which sends each coordinate of $x \in \Z^n$ to its class
modulo $p$. Any element $S \in \Gr_{n,r}(\F_p)$ gives rise to a sub-lattice
$\pi_p^{-1}(S) \subset \Z^n$, which contains $p\Z^n$ as a
subgroup of index $p^{r}$, and with $\pi_p^{-1}(S) / p\Z^n$ isomorphic
as an abelian group to $S \cong \prod_1^r \Z/p\Z$. Similarly, for any
$L' \in \Lambda_{p,r}(\Z^n)$ we have  $S
= \pi_p(pL') \cong L'/\Z^n$. This shows that for any lattice $L = g\Z^n$ we have
$$\Lambda_{p,r}(L) = \{p^{-1} g \pi_p^{-1}(S) : S \in
\Gr_{n,r}(\F_p)\}.$$
We have shown:


\begin{prop}\name{prop: still haar}
Choosing $L'$ according to $\mu_{n, p^{-r}}$ is the same as choosing
$L = g\Z^n$ according to $\mu_n$, then choosing $S \in \Gr_{n,r}(\F_p)$
uniformly and setting $L' = p^{-1} g \pi_p^{-1}(S).$ 
\end{prop}

We can state Proposition \ref{prop: still haar} in more concrete terms
as follows. Choose a random lattice $L$
distributed according to $\mu_n$, choose generators $v_1, \ldots, v_n$
of $L$, so that the parallelepiped 
\eq{eq: parallelepiped}{
  \PP_L = \left\{\sum a_i v_i : \forall i, \ 0 \leq a_i < 1 \right \}
  }
is a fundamental
domain for $\R^n/L$. Define the discrete `net' 
\eq{eq: discrete par}{\PPdisc \df \left\{
\sum a_i v_i \in \PP_L:  a_i \in  \left\{ 0, \frac{1}{p}, \ldots, 1-\frac{1}{p}
\right \} \right \}. 
}
These are coset representatives for the inclusion $L \subset
\frac{1}{p} \cdot L$. Choose 
 elements $w_1, \ldots, w_r \in \PPdisc$ from the uniform distribution over linearly independent (as elements of $\F_p^n$)
 $r$-tuples. \comor{changed the phrasing a bit above.}
 Then the
      lattice $L' = \spa_{\Z}\left(v_1, \ldots, v_n, w_1, \ldots,
      w_r\right)$ is a random lattice distributed according to
      $\mu_{n, p^{-r}}$.

\subsection{Some bounds of Rogers and Schmidt}
We now recall some fundamental results of Rogers and
Schmidt. For a
lattice $L \in \LL_n$ let 
$\TT_L \df \R^n/L$ be the quotient torus, let $m_L$ be the Haar probability
measure on $\TT_L$, and let $\pi_L : \R^n \to
\TT_L$ be the quotient map. Let $\Vol ( \cdot)$ denote the Lebesgue measure on
$\R^n$. For a Borel measurable subset $J \subset \R^n$, and a lattice
$L \subset \R^n$, let
$$\vre(J, L) \df 1-m_L\left(\pi_L(J) \right);$$
equivalently, 
$\vre(J, L)$ is the density of points in $\R^n$ not covered by $L
+J$. Also let
\eq{eq: def eta}{
\eta = \eta_n \df \frac{n}{4} \log \left( \frac{27}{16} \right) - 3
\log n. 
}
With these notations, the following was shown in \cite{Rogers_bound}
(see also \cite{Schmidt-admissible}):
\begin{thm}\name{thm: Rogers bound}
  There is a positive constant $\crog$ such that for all $n \in \N$,
  for every Borel measurable $J \subset \R^n$
  with
  $$V \df \Vol(J) \leq \eta$$
  we have
  $$\left|   \int_{\LL_n} \vre(J, L) d\mu_n(L) - e^{-V}
    \right | < \crog \cdot e^{-\eta}.$$

  \end{thm}

  Using the Markov inequality, this immediately implies the following:

  \begin{cor}\name{cor: Rogers bound}
    With the same notation and assumptions, for any $\kappa>0$, 
    \eq{eq: conclusion chebyshev}{
      \mu_n
      \left( \left\{ L \in \LL_n: \vre(J, L) > \kappa 
  \right \} \right) < \frac{1}{\kappa} \left(e^{-V}
  + \crog \,
e^{-\eta}
  \right) . 
}
\end{cor}


\subsection{From half covering to full covering.}
Here we show the standard fact (cf. \cite[Lemma 4]{Rogers_again} or \cite{HLR09}) that if a convex body 
\comor{I changed "body" to "convex body"} covers half the space, then dilating it by a factor 2 covers all of space. 
Notice that this translates to a factor $2^n$ in volume, as a result of which we will only use this lemma
for very small bodies. \combarak{Oded, isn't there a reference for this? You told me it is in one of your papers. Or was it only written there for the Euclidean ball?}\comoded{We have it in~\cite{HLR09} for all $\ell_p$ balls, so it's not quite the same. I bet it's pretty standard, but I was unable to find a good reference. Up to you... }

\begin{lem}\name{lem: half to full}
Let $\KK \in \Convn$ and let $L$ be a lattice in $\R^n$.
Suppose that
\[
m_{L}(\pi_{L}(\KK))>\frac{1}{2}.
\]
Then we have
\[
L + 2\KK = \R^n.
\]
\end{lem}
\begin{proof}
Since 
\[
  m_{L}\left(\pi_{L}\left(\KK \right)\right) 
  = 
  m_{L}\left(\pi_{L} \left( -\KK \right) \right)
  >
  \frac{1}{2} \; ,
\]
we have that for any $x \in \TT_{L}$,
\[
m_{L}\left( \left(\pi_{L}\left(\KK \right) -x \right) \cap
  \left(\pi_{L} \left(-\KK \right) \right) \right) > 0 \; .
\]
Therefore, there are $z_1, z_2 \in \pi_{L}(\KK)$ so that $z_1 -x=
-z_2$, or equivalently, there are $y_1, y_2 \in \KK$ so that  
\[
 x =  \pi_{L}\left( y_1 \right) +
\pi_{L}\left( y_2\right) = 
 \pi_{L}\left(  y_1 + y_2\right).
 \]
The claim now follows from $y_1+y_2 \in 2\KK$.
\end{proof}

\subsection{Lower bound on the size of a discrete \texorpdfstring{$\vre$}{eps}-Kakeya
  set}\name{subsec: Kakeya}
Now let $q$ be a power of a prime, let $\F_q$ denote the field with
$q$ elements and for a line $\ell 
\in \Gr_{n,1}(\F_q)$, let $x + \ell$ denote the affine line through
$x$ parallel to $\ell$.  
A subset $K\subset \F_q^n$ is called a {\em Kakeya set} if for
every $\ell \in \Gr_{n,1}(\F_q)$ there is $x \in \F_q^n$ such that
$x + \ell  \subset K;
$
that is, $K$ contains a line in every direction. For $\vre \in (0,1]$, $K$ is
called an {$\vre$-Kakeya set} if
$$
\left|\left\{\ell  \in \Gr_{n,1}(\F_q): \exists x \text{ s.t. } x +
\ell \subset K
\right\} \right|  \geq \vre \, |\Gr_{n,1}(\F_q)|;
$$
that is $K$ contains a line in at least an $\vre$-proportion of
directions.
Extending this notion to higher dimensions, let $\vre \in (0,1]$ and
$r \in \{1, \ldots, 
n -1 \}$. Then a set  $K \subset \F_q^n$ is called a {\em Kakeya set of
  rank $r$} if for any $S \in \Gr_{n,r}(\F_q)$ there is $x \in \F_q^n$
such that $x + S \subset K$, and an {\em $\vre$-Kakeya set of rank
  $r$} if
$$
\left|\left\{S  \in \Gr_{n,r}(\F_q): \exists x \text{ s.t. } x +
S \subset K
\right\} \right|  \geq \vre \, |\Gr_{n,r}(\F_q)|. 
$$



In this subsection we will derive lower bounds on the size of an $\vre$-Kakeya set of rank $r$. Our main observation is that the possible sizes of an $\vre$-Kakeya set and a $\delta$-Kakeya set are related as follows.
\combarak{Rearranged this section as Or suggested, removing the multiple versions of Lemma \ref{lem:weakdeltabound}, once with $\delta=1$ and once with general $\delta$.}
\begin{lem}\name{lem:epstodelta}
Let $0<\vre<\delta<1$. Assume $K\subset \F_q^n$ is an $\vre$-Kakeya set of rank $r$, then there exists a $\delta$-Kakeya set $\AAA\subset \F_q^n$ of rank $r$ with cardinality
$$|\AAA|\leq \left\lceil\frac{\log(1-\delta)}{\log(1-\vre)} \right\rceil|K|.$$
\end{lem}

\begin{proof}
	Fix $n
	\in
	\N$ and $r
	\in \{1, \ldots, n-1\}$. For $K \subset \F_q^n$, denote
	$$
	\BB_K \df \{S \in \Gr_{n,r}(\F_q) : K \text{ contains a translate of }
	S \}.
	$$
	Let $g$ be an element of $\GL_n(\F_q)$, that is an invertible
	$n\times n $ matrix
	with entries in $\F_q$. For $S \in \Gr_{n,r}(\F_q),$ we clearly have
	$S \in \BB_K$ if and only if $gS \in \BB_{gK}$. 
	Let $\NN$ be a finite subset of $\GL_n(\F_q)$ and consider 
	$$
	\AAA = \AAA(\NN , K) \df \bigcup_{g \in \NN} g K.
	$$
	Clearly
	$$|\AAA| \leq
	|\NN|\cdot|K| \ \ \text{ and } 
	\bigcup_{g \in \NN} g \BB_K   \subset \BB_{\AAA}.
	$$
%
%
Recall that by definition of an $\vre$-Kakeya set of rank $r$, we have that $|\BB_K|\geq \vre |\Gr_{n,r}(\F_q)|$.
	Consequently, our claim will follow once we show that if
	\eq{eq: if}{
		\BB \subset\Gr_{n,r}(\F_q) \text{ satisfies } |\BB|\geq \vre \,
		|\Gr_{n,r}(\F_q)|}
	then
	\eq{eq: then}{
		\exists \,  \NN\subset\GL_n(\F_q) \text { s.t. } |\NN| \leq 
		\left\lceil\frac{\log(1-\delta)}{\log(1-\vre)} \right\rceil 
		\text{ and } 
		\Big|\bigcup_{g \in \NN} g \BB \Big| \geq \delta |\Gr_{n,r}(\F_q)|.}   
	
	We will prove this using a standard probabilistic argument. 
	Define a probability space by drawing $N=\left\lceil\frac{\log(1-\delta)}{\log(1-\vre)} \right\rceil$ elements $g_1, \ldots, g_N$ of
	$\GL_n(\F_q)$, uniformly 
	and independently.
	Fix $\BB$ as in \equ{eq: if}  and for each $S \in \Gr_{n,r}(\F_q)$, denote by 
	$E^i_S$ the event that $S \notin g_i \BB$.
	The events $\{E^i_{S} : i=1, \ldots, N \}$ are i.i.d., since
	the $g_i$ are. Therefore 
	$$E_S  \df \bigcap_{i=1}^N E^i_S ,$$
	satisfies 
	$$
	\prob\left(E_{S}\right) 
	=\left(\prob(E^i_{S})\right)^N.
	$$
	Since $\GL_n(\F_q)$ acts transitively on
	$\Gr_{n,r}(\F_q)$, 
	$$
	\prob\left(E^i_S \right) = \prob\left(g_i^{-1} S \notin
	\BB\right)=1-\frac{|\BB|}{|\Gr_{n,r}(\F_q)|}\leq 1-\vre, 
	$$
	which implies
	\eq{eq: which implies}{
		\prob\left(E_{S} \right)\leq \left(1-\vre\right)^N 
		\leq 1-\delta.
	}
It therefore follows that
\[
\mathbb{E}\Big|\bigcup_{i=1}^N g_i \BB \Big|
=
\sum_{S\in\Gr_{n,r}(\F_q)}\left(1-\prob\left(E_{S} \right)\right)
\geq 
\delta |\Gr_{n,r}(\F_q)|.
\]
%
	This implies that there exists a subset $\NN 
	\df \{g_1, \ldots, g_N\}$ that
	satisfies \equ{eq: then}.
\end{proof}

In~\cite{Dvir, DKSS, KLSS}, a fundamental lower bound on the minimal cardinality of  Kakeya sets was established. We will need 
the following variant, whose special case $\delta=1$ was proved in \cite{KLSS}: 


\begin{lem}\name{lem:weakdeltabound}
Let $\delta\in(0,1]$. If $K \subset \F_q^n$ is a $\delta$-Kakeya set of rank $r$ then
$$|K| \geq \left(1 +
\frac{(q-1)q^{-r}}{\delta} \right)^{-n}q^n .$$ 
\end{lem}
The proof follows with minor adaptations from the arguments of \cite{KLSS}. We give the details in 
Appendix \ref{appendix A}. 

The bound in Lemma \ref{lem:weakdeltabound} is quite tight for large $\delta$, but is loose for $\delta\ll 1$. We now leverage Lemma~\ref{lem:epstodelta} to obtain a much sharper bound for small $\delta$. It 
replaces the exponential (in $n$, with $q,r$ fixed) dependence on $\delta$ with a linear dependence.
We remark that the bound \eqref{eq:r1} will not be used in this paper, and is included for future reference.
\begin{thm}\name{thm: bound eps Kakeya}
Let $\vre\in(0,1)$. If $K\subset\F_q^n$ is an $\vre$-Kakeya set of rank $r$, then 
\eq{eq:generalr}{|K|>\vre \left(1 +2
(q-1)q^{-r}\right)^{-n} q^n,}
and for $r=1$ we also have that
\eq{eq:r1}{
|K|>\vre \frac{e^{-1}}{\log(2e n)}2^{-n} q^n.}
\end{thm}

\begin{proof}

We first claim that if $K\subset\F_q^n$ is an $\vre$-Kakeya set of rank $r$, then for any $\vre<\delta<1$
\eq{eq:generaldelta}{|K|\geq \left(\left\lceil\frac{\log(1-\delta)}{\log(1-\vre)}\right\rceil\right)^{-1}\cdot\left(1 +\frac{1}{\delta}
(q-1)q^{-r}\right)^{-n} q^n.}
To see this, let $\vre<\delta<1$, and assume for contradiction that $K\subset \F_q^n$ is an $\vre$-Kakeya set of rank $r$ with cardinality smaller than the right-hand side of~\eqref{eq:generaldelta}. By Lemma~\ref{lem:epstodelta}, this implies that there must exist a $\delta$-Kakeya set $\AAA\in\F_q^n$ of rank $r$ with cardinality $|\AAA|\leq  \left\lceil\frac{\log(1-\delta)}{\log(1-\vre)}\right\rceil|K|< \left(1 +\frac{1}{\delta}
(q-1)q^{-r}\right)^{-n} q^n$, which contradicts Lemma~\ref{lem:weakdeltabound}.

Next, we use~\eqref{eq:generaldelta} to show that if $|K|$ is an $\vre$-Kakeya set of rank $r$, it must satisfy~\eqref{eq:generalr}.
For $\vre\in[1/2,1)$, this follows immediately from Lemma~\ref{lem:weakdeltabound}. We may therefore assume without loss of generality that $\vre\in(0,1/2)$. Let $\delta=1/2$ and note that for all $\vre$ in this range
$$\left\lceil\frac{\log(1-\delta)}{\log(1-\vre)} \right\rceil= \left\lceil\frac{\log{(2})}{-\log(1-\vre)} \right\rceil<\frac{\log{(2})}{-\log(1-\vre)}+1< \frac{1}{\vre}.$$
Thus, applying~\eqref{eq:generaldelta} with $0<\vre<\delta=1/2$ establishes~\eqref{eq:generalr}.

Finally, we assume $r=1$ and establish~\eqref{eq:r1}. 
Let $\delta=\frac{1}{1+2/n}$ and note that 
\[
\left(1 +\frac{1}{\delta} (q-1)q^{-r}\right)^{-n} q^n
\geq 
e^{-1}2^{-n}q^n \; .
\]
Hence, for $\vre\in\left[\delta,1\right]$, \eqref{eq:r1} follows immediately from Lemma~\ref{lem:weakdeltabound}.  
We may therefore assume without loss of generality that $\vre\in\left(0,\delta \right)$. 
For all $\vre$ in this range
$$\left\lceil\frac{\log(1-\delta)}{\log(1-\vre)} \right\rceil= \left\lceil\frac{\log{\left(1+\frac{n}{2}\right)}}{-\log(1-\vre)} \right\rceil<\frac{\log\left(2n\right)}{\vre}+1< \frac{\log(2en)}{\vre}.$$
Thus, applying~\eqref{eq:generaldelta} with $0<\vre<\delta=\frac{1}{1+2/n}$ establishes~\eqref{eq:r1}.
\end{proof}

%

We will need the following consequence:

\begin{cor}\name{cor: bound eps Kakeya}
  \begin{itemize}
    \item[(i)]
If $K \subset \F_q^n$ is an $\vre$-Kakeya set of rank 2 then $\frac{|K|}{q^n} >
\vre e^{-2n/q}.$
\item[(ii)]
If $K' \subset
\F_q^n$ satisfies $\frac{|K'|}{q^n} \ge 1 -\vre e^{-2n/q}$
then the set
$$
\SSS \df \{S \in \Gr_{n,2}(\F_q) : \forall x \in \F_q^n, \, (x+S) \cap K' \neq \varnothing\}
$$
satisfies
$$
|\SSS| > (1-\vre) |\Gr_{n,2}(\F_q)|. 
$$
\end{itemize}
  \end{cor}
\begin{proof}
 For $r=2$ we  have that
  $$
\frac{1}{\left(1 +2
	(q-1)q^{-r}\right)^{-n}}=\left(1+\frac{2}{q}-\frac{2}{q^2}\right)^{n}\leq
    \left[\left(1+\frac{2}{q}\right)^{q/2}\right]^{2n/q}\leq e^{2n/q}.
    $$
Thus (i) is an immediate consequence of~\eqref{eq:generalr}. Assertion (ii) follows from (i) by setting
$K = \F_q^n \sm K'$. 
\end{proof}

\section{Proof of Theorem~\ref{thm: main}}

Let $p$ be a prime number satisfying 
\eq{eq: proved it again}{
  n \leq p \leq 2 n.
}
We define a probability space as follows.
Let $L=g\Z^n$ be a random lattice chosen according to $\mu_n$, and $S$ be randomly chosen from the uniform distribution on $\Gr_{n,2}(\F_p)$, independently of $L$. Define the lattice $L'=p^{-1}g\pi_p^{-1}(S)$ and note that $L\subset L'\subset \frac{1}{p}L$. By Proposition~\ref{prop: still haar}, we have that $L'$ is distributed according to $\mu_{n,p^{-2}}$. 
Therefore, the left-hand side of~\eqref{eq: conclusion main}, which we are trying to bound from above, is equal to 
\[
\Pr\Big(\Theta_{\KK}(L)> M \Big) = 
\Pr\Big(\Theta_{\KK}(L')> \frac{M}{p^2} \Big) \; .
\]

Let $J$ be the dilate of $\KK$ of volume 
\eq{eq: def V}{V \df p^{-2} \left(1+\frac{2}{p} \right)^{-n} M.}
Applying Corollary~\ref{cor: Rogers bound} with $\kappa = e^{-\frac{V}{2}}$ we have
\[
\Pr\Big( \vre(J,L)> e^{-\frac{V}{2}} \Big) <  c_0 \, e^{-\frac{V}{2}},
\]
where
$c_0 \df 1 + \crog$. 
Here we used that $V \le \eta$ (where $\eta$ is as defined in~\eqref{eq: def eta}) which holds assuming the constant $c_4$ is chosen small enough. 
From now on, we fix an $L$ for which 
\begin{align}\label{eq:bound on hole}
   \vre(J,L) \le e^{-\frac{V}{2}},    
\end{align}
and we show that when choosing $S $, with probability at least $1- \vre$, for $\vre$ to be chosen below, we have $\Theta_{\KK}(L') \le M/p^2$. 
\combarak{tweaked a bit in the preceding paragraph, so that the numerology of constants gets explained at the end. Right now one has to look at two places, here and at the end of the proof. Feel free to undo if you think it makes things worse.}

Define
\eq{eq: def BL}{
  B_L \df \TT_L \sm \pi_L(J),
}
so that $m_L (B_L) \leq  e^{-V/2}$. 
\comoded{do we really need both notations $\PPdisc$ and $\overline{\mathcal{P}}_L^{\mathrm{(disc)}}$? The notation feels a bit heavy, especially given the short proof}
Let $\PPdisc$ be as in \equ{eq: discrete par}, and let 
\eq{eq: ppdiscbar}{
  \overline{\mathcal{P}}_L^{\mathrm{(disc)}}  \df \pi_L\left(\PPdisc \right)  =
\pi_L \left(\frac{1}{p} \cdot L  \right) \subset
\TT_L.
}
The Haar measure $m_L$ on the torus $\TT_L$
satisfies 
that
$$
m_L(A) = \frac{1}{p^n}\sum_{x \in \PPdiscbar} m_L(A-x),
$$
and applying this with $A$ taken to be $B_L$ we find that 
there is $u \in \TT_L$ such that
$$
\frac{1}{p^n} \left|\left(u + \PPdiscbar \right) \cap B_L \right|
\leq m_L(B_L) \leq  e^{-\frac{V}{2}} \;.
$$
Recall that we have an identification of $\F_p^n$
with $\big( 0, \frac{1}{p},
    \ldots, 1-\frac{1}{p} \big)^n$ by reducing mod $p$ and then
  dividing by $p$, and a further identification of $\big( 0, \frac{1}{p},
    \ldots, 1-\frac{1}{p} \big)^n$ with $\PPdiscbar$. With these
  identifications in mind we view $\F_p^n$ as a subset of $\TT_L$, and
  define 
$$
K' \df \left\{x \in \F_p^n : u + x \in \pi_L(J) \right\},
$$
so that 
$$\frac{|K'|}{p^n} \ge 1-e^{-\frac{V}{2}}.$$  
This implies via Corollary \ref{cor: bound eps Kakeya}(ii), applied with \eq{eq: defn epsilon1}{
\vre\df e^{-\frac{V}{2}}e^{2n/p} ,
}that 
with probability at least $1-\vre$ over the choice of $S$, it holds that for all $x \in \F_p^n$, 
$u+ x +S$ intersects $\pi_L(J)$. 
Recalling that $L'=\frac{1}{p}\cdot g\pi_p^{-1}(S)$, this equivalently says that
\begin{align}\label{eq: cover the grid}
 u + \frac{1}{p} \cdot L \subset L' + J \; .
\end{align}
But by Lemma~\ref{lem: half to full} and \eqref{eq:bound on hole}, and using that $V > 2 \log 2$ (which we can assume by taking $c_3$ large enough), we have
$$
 L+2J = \R^n . 
$$
Together with~\eqref{eq: cover the grid}, this implies that
\begin{align*}
 L' + \Big(1+\frac{2}{p}\Big) J  
 \supset 
 u + \frac{1}{p} \cdot L + \frac{2}{p} J 
 =
 \R^n
 \; .
\end{align*}%
To summarize, Proposition \ref{prop: still haar} shows that with all but probability $c_0 \, e^{-V/2} + \vre$ (due to the choice of $L$ and $S$), we have 
$L' + \Big(1+\frac{2}{p}\Big) J  = \R^n$ and hence
$\Theta_{\KK}(L') \le \frac{M}{p^2}$. Using our choices \equ{eq: proved it again}, \equ{eq: def V} and \equ{eq: defn epsilon1} we see that for appropriate choices of constants $c_1, c_2,$
we have \equ{eq: conclusion main}. 

\appendix
\section{Proof of Lemma~\ref{lem:weakdeltabound}}\name{appendix A}
The case $\delta=1$ is precisely~\cite[Theorem 1]{KLSS}. The general
case $\delta \in (0,1]$ (as in Lemma~\ref{lem:weakdeltabound}) follows 
from minor modifications to their proof. For the reader's convenience, we include
the full proof here, much of it taken verbatim from~\cite{KLSS}. 

We start with some necessary background. Let $\N_0$ denote the set of non-negative integers. For an $n$-tuple $i=(i_1,\ldots,i_n) \in \N_0^n$, we define $\|i\| \df i_1+\cdots +i_n$ and if $X=(X_1,\ldots,X_n)$ then $X^i \df X_1^{i_1} \cdots X_n^{i_n}$. Any polynomial $P$ in $n$ variables over some field $\F$ can be expanded in the form
\[
P(X+Y) = \sum_{i \in \N_0^n} P^{(i)} (Y) X^i \; ,
\]
for some polynomials $P^{(i)}$ over $\F$ in $n$ variables. We refer to $P^{(i)}$ as the \emph{Hasse derivative of $P$ of order $i$}. It is easy to see that $P^{(0)} = P$ and that for $\|i\| > \deg P$, $P^{(i)} = 0$. Moreover, if $P_H$ denotes the homogeneous part of $P$, then
  \[
    (P_H)^{(i)} = \left\{\begin{array}{lr}
        (P^{(i)})_H & \text{if } \deg P^{(i)} = \deg P - \|i\| ,\\
        0 & \text{if } \deg P^{(i)} < \deg P - \|i\| .
        \end{array}
        \right.
  \]

For a nonzero polynomial $P$ in $n$ variables over a field $\F$, 
we define its \emph{multiplicity of zero} at some point $a \in \F^n$, 
denoted $\mu(P,a)$, as the largest $m \ge 0$ such that $P^{(i)}(a) = 0$ for 
all $i \in \N_0^n$ with $\|i\| < m$. Alternatively, it is the largest $m$ for which we can write
\[
P(X+a) = \sum_{i \in \N_0^n ~:~ \|i\| \ge m} c(i,a) X^i
\]
for some $c(i,a) \in \F$. We sometimes also say that $P$ \emph{vanishes} at $a$ with multiplicity $m$. 

We will use the following relatively straightforward lemmas.

\begin{lem}[{{\cite[Lemma 5]{DKSS}}}]\label{lem:mudecrease}
Let $n \ge 1$ be an integer. 
For any nonzero polynomial $P$ in $n$ variables over a field $\F$, $a \in \F^n$, and $i \in \N_0^n$, it holds that
\[
\mu(P^{(i)},a) \ge \mu(P,a) - \|i\| \; .
\]
\end{lem}

\begin{lem}[{\cite[Proposition 10]{DKSS}}]\label{lem:existenceofpoly}
Let $n,m \ge 1$ and $k \ge 0$ be integers, and $\F$ a field.
If a finite set $S \subset \F^n$ satisfies $\binom{m+n-1}{n} |S| < \binom{n+k}{n}$,
then there exists a nonzero polynomial over $\F$ in $n$ variables of degree at most $k$,
vanishing at every point of $S$ with multiplicity at least $m$.
\end{lem}

\begin{lem}[{\cite[Lemma 14]{KLSS}}]\label{lem:polynomialrestriction}
Let $n,r \ge 1$ be integers, and $P$ a nonzero polynomial in $n$ variables over a field $\F$. 
Suppose that $b,d_1,\ldots,d_r \in \F^n$. Then for any $t_1,\ldots,t_r \in \F$, 
\[
\mu(P(b+T_1 d_1 + \cdots + T_r d_r), (t_1,\ldots,t_r)) \ge \mu(P,b+t_1 d_1 + \cdots + t_r d_r) \; ,
\]
where we view $P(b+T_1 d_1 + \cdots + T_r d_r)$ as a polynomial in the formal variables $T_1,\ldots, T_r$. 
\end{lem}

\comoded{5/28: gave more credit to DKSS as suggested by Barak}
Finally, we will need a multiplicity version of the standard Schwartz-Zippel lemma~\cite{DKSS}. 

\begin{lem}[{\cite[Lemma 2.7]{DKSS}}]\label{lem:schwartzzippel}
Let $n \ge 1$ be an integer, $P$ a nonzero polynomial in $n$ variables over a field $\F$, and $S \subset \F$ a finite set. Then
\[
|S|^{-(n-1)} \sum_{z \in S^n} \mu(P,z) \le \deg P \; .
\]
\end{lem}

\begin{proof}[Proof of Lemma~\ref{lem:weakdeltabound}]
Let $m,k$ be positive integers satisfying 
\begin{align}\label{eq:choice of m and k}
k < \delta q^r \left\lceil \frac{qm-k}{q-1} \right\rceil \; .
\end{align}
Our goal for the rest of the proof is to show that under the condition~\eqref{eq:choice of m and k},
\begin{align}\label{eq:lower bound on K}
|K| \ge \binom{n+k}{n} \Big/ \binom{m+n-1}{n} \; .
\end{align}
The lemma then follows by taking $k=N q^{r+1} - 1$ and $m=\lceil (q^r + \frac{q - 1}{\delta})N \rceil$
where $N$ is a positive integer. With this choice,~\eqref{eq:choice of m and k} holds, and
the lemma follows by noting that the right-hand side of~\ref{eq:lower bound on K} converges
to $\left(1 + (q-1)q^{-r}/\delta \right)^{-n}q^n$ as $N$ goes to infinity. 

Assume towards contradiction that~\eqref{eq:lower bound on K} does not hold. Thus, by Lemma~\ref{lem:existenceofpoly}, there exists a nonzero polynomial $P$ in $n$ variables over $\F_q$ of degree at most $k$ that vanishes at every point of $K$ with multiplicity at least $m$. 
Let $\ell \df \lceil \frac{qm-k}{q-1} \rceil$ and fix $i=(i_1,\ldots,i_n) \in \N_0^n$ satisfying $w \df \|i\| < \ell$. Let $Q \df P^{(i)}$ be the $i$th Hasse derivative of $P$.

Let $D \subset (\F_q^n)^r$ be the set of all $r$-tuples of vectors $(d_1,\ldots,d_r)$
with the property that there exists $b \in \F_q^n$ such that
$b + t_1 d_1 + \cdots + t_r d_r \in K$ for all $t_1,\ldots, t_r \in \F_q$.
Since $K$ is a $\delta$-Kakeya set of rank $r$, we have
$|D| \ge \delta \cdot q^{nr}$.
(Notice that the span of $d_1,\ldots,d_r$ might be of rank less than $r$;
the statement is true because a $\delta$-Kakeya set of rank $r$ is also a
$\delta$-Kakeya set of rank $r'$ for all $r' \le r$.)\comoded{5/21: is that sufficiently clear? it's an annoying issue, and I was trying not to waste too much time on it} \combarak{I agree it's true and I agree it's a little annoying.}
Therefore, by our choice of $P$, for any $(d_1,\ldots,d_r) \in D$, there exists a $b \in \F_q^n$
such that for all $t_1,\ldots, t_k \in \F_q$,
\[
\mu(P, b+t_1 d_1 + \cdots + t_r d_r) \ge m \; ,
\]
and so by Lemma~\ref{lem:polynomialrestriction} and Lemma~\ref{lem:mudecrease},
\begin{align*}
\mu(Q(b+T_1 d_1+\cdots+T_r d_r), (t_1,\ldots,t_r)) &\ge
\mu(Q, b+t_1 d_1 + \cdots + t_r d_r) \\
& \ge m - w\; ,
\end{align*}
where in the left-hand side we consider 
$Q(b+T_1 d_1+\cdots+T_r d_r)$ as a polynomial in the variables
$T_1,\ldots,T_r$. 
But since
\[
\deg Q(b+T_1 d_1+\cdots+T_r d_r) \le \deg Q \le k-w < q(m-w) \; 
\]
(which follows from $w < \ell$), 
Lemma~\ref{lem:schwartzzippel} with $S=\F_q$ implies that $Q(b+T_1 d_1+\cdots+T_r d_r)$ 
is in fact the zero polynomial. 
 
Let $P_H$ and $Q_H$ denote the homogeneous parts of $P$ and $Q$, respectively (i.e., $P_H$ is the unique homogeneous polynomial for which $\deg (P-P_H) < \deg P$). 
It is easy to see that $Q(b+T_1 d_1+\cdots+T_r d_r)=0$ implies 
$Q_H(T_1 d_1+\cdots+T_r d_r)=0$ (note that there is no $b$ in the latter). 
It follows that $(P_H)^{(i)}(T_1 d_1+\cdots+T_r d_r)=0$
for all $(d_1,\ldots,d_r) \in D$.
Equivalently, $(P_H)^{(i)}$, considered as a polynomial in $n$ variables 
over the field of rational functions $\F_q(T_1,\ldots,T_r)$, vanishes at every point 
of the set
\[
D' \df \big\{ T_1 d_1 + \cdots + T_r d_r ~:~ (d_1,\ldots,d_r) \in D \big\} \subset S^n \; ,
\]
where
\[
S \df \big\{ \alpha_1 T_1 + \cdots + \alpha_r T_r ~:~ \alpha_1,\ldots,\alpha_r \in \F_q \big\} \; .
\]
Since $i$ is an arbitrary tuple satisfying $\|i\| < \ell$, 
this shows that $P_H$ vanishes with multiplicity at least $\ell$ at every point of $D'$.
On the other hand, by~\eqref{eq:choice of m and k},
\[
\deg P_H = \deg P \le k < \delta q^r \ell = \delta |S| \ell \; ,
\]
which implies by Lemma~\ref{lem:schwartzzippel} that $P_H$ is the zero polynomial. 
This is a contradiction since the homogeneous part of a nonzero polynomial is nonzero. 
\end{proof}





\section{The expectation of the covering density is infinite}\name{appendix B}

\begin{prop}
There is $c>0$ such that for all $n$ large enough and all $M \geq 1$ we have \eq{eq: conclusion main2}{
\mu_n \left( \left\{ L \in \LL_n: \Theta(L) > M \right \}\right) \geq \frac{c \, V_n^2}{2^n} \, \frac{1}{M} \; ,
}
where $V_n$ denotes the volume of the Euclidean ball of radius one in $\R^n$.
In particular, for any $n$ and any $\KK \in \Convn$ there is $C>0$ such that 
$$
\mu_n\left( \left\{L \in \LL_n : \Theta_{\KK}(L)>M \right \} \right) > \frac{C}{M}.
$$
\end{prop}
\begin{proof}
Let $\lambda_1(L)$ denote the length of the shortest nonzero vector of $L$.
Given $M$, let $r \df (M/V_n)^{1/n}$ be the radius of a Euclidean ball of volume $M$.
Let $L^*$ denote the dual lattice of $L$, that is 
\eq{eq: dual lattice}{L^* \df \{u \in \R^n: \forall v \in L, \, u \cdot v \in \Z\}.
}
By considering the distance between affine hyperplanes perpendicular to the shortest nonzero vector of $L^*$, we see that $\lambda_1(L^*) < \frac{1}{2r} 
$ implies that $\Theta(L)>M$. In particular, taking into account that the measure $\mu_n$ is invariant under the mapping $\LL_n \to \LL_n, \ L \mapsto L^*$, we see that the left-hand side of \equ{eq: conclusion main2} is bounded below by 
\eq{eq: bound from below}{
  \mu_n \left( \left\{ L \in \LL_n: \lambda_1(L)< \frac{1}{2r} \right\} \right).}
  Using Siegel's summation formula, Kleinbock and Margulis \cite[\S 7]{Kleinbock_Margulis} obtained the estimate
  $$
  \mu_n \left( \left\{ L \in \LL_n : \lambda_1(L) < t \right\} \right) \geq \frac{1}{2 \zeta(n)} V_n t^n - \frac{1}{4 \zeta(n-1)  \zeta(n)} V_n^2 t^{2n},
  $$
  where $\zeta(n) = \sum_{m\in \N} m^{-n}$ is the Riemann zeta function. 
  Applying this estimate with $t = \frac{1}{2r}$ and using $M = V_n r^n$ and the fact that $\zeta(n) \to_{n\to \infty} 1$, we get that the left-hand side of \equ{eq: conclusion main2} is bounded from below by
  $$\frac{V_n^2}{2^{n+2}M} - \frac{V_n^4}{2^{2n+2}M^2}.
  $$
  By standard estimates for $V_n$, the first summand in this expression is the dominant one for $M \geq 1$. 
  
 This proves \equ{eq: conclusion main2}. Since any $\KK \in \Convn$ is contained in a dilate of a Euclidean ball, the second assertion of the proposition follows. 
\end{proof}

\bibliographystyle{alpha}
\bibliography{elonbib}

\end{document}